\documentclass{amsart}
\usepackage{amssymb}
\usepackage{mathtools}
\usepackage[svgnames]{xcolor}
\usepackage[unicode,
  colorlinks=true,
  linktocpage=true,
  citecolor=OliveDrab,
  linkcolor=DarkMagenta,
  linkcolor=DarkMagenta,
  pdftitle={Embeddable partial groups},
  pdfauthor={Philip Hackney, Justin Lynd, and Edoardo Salati}
  ]{hyperref}
\usepackage[capitalise,noabbrev]{cleveref}
\usepackage{tikz-cd}
\usepackage{not_clever}
\usetikzlibrary{calc,shapes.geometric,decorations.pathmorphing}
\usepackage{microtype}
\usepackage[T1]{fontenc}

\newcommand{\op}{\textup{op}}
\newcommand{\set}{\mathsf{Set}}
\newcommand{\sset}{\mathsf{sSet}}
\newcommand{\sym}{\mathsf{Sym}}
\newcommand{\cat}{\mathsf{Cat}}
\newcommand{\gpd}{\mathsf{Gpd}}

\newcommand{\words}{\mathbf{W}}
\newcommand{\bswords}{\mathbf{S}}

\DeclareMathOperator{\id}{id}
\DeclareMathOperator{\sk}{sk}

\newcommand{\leftsquigarrow}{\mathrel{\reflectbox{$\rightsquigarrow$}}}

\newcommand{\rep}[1]{\Upsilon^{#1}}
\newcommand{\spine}[1]{\textup{Sp}^{#1}}
\newcommand{\circspine}[1]{\textup{cSp}^{#1}}
\newcommand{\NA}{\operatorname{NA}}
\newcommand{\NNA}{\operatorname{A}}

\usepackage{mathrsfs}
\newcommand{\edgemap}{\mathscr{E}}
\newcommand{\bous}{\mathscr{B}}

\DeclareMathOperator{\rs}{\mathsf{rs}}

\newtheorem{theorem}{Theorem}
\newtheorem{corollary}[theorem]{Corollary}
\newtheorem{proposition}[theorem]{Proposition}
\newtheorem{lemma}[theorem]{Lemma}

\theoremstyle{definition}
\newtheorem{definition}[theorem]{Definition}
\newtheorem{example}[theorem]{Example}

\theoremstyle{remark}
\newtheorem{remark}[theorem]{Remark}

\begin{document}

\title{Embeddable partial groups}

\author{Philip Hackney}
\address{Department of Mathematics, University of Louisiana at Lafayette}
\email{philip@phck.net} 
\urladdr{http://phck.net}

\author{Justin Lynd}
\address{Department of Mathematics, University of Louisiana at Lafayette}
\email{lynd@louisiana.edu}

\author{Edoardo Salati}
\address{Department of Mathematics, RPTU University Kaiserslautern-Landau}
\email{edoardo.salati@rptu.de}

\thanks{
This work was supported by a grant from the Simons Foundation (\#850849, PH)
and by a grant from the Simons Foundation International (SFI-MPS-TSM-00014188, JL). 
PH was partially supported Louisiana Board of Regents through the Board of Regents Support fund LEQSF(2024-27)-RD-A-31.
ES is grateful for travel support from TU Nachwuchsring at the RPTU Kaiserslautern-Landau.
}

\begin{abstract}
We record a folklore theorem that says a partial group embeds in a group if and only if each word has at most one possible multiplication, regardless of choice of parenthesization. 
We further investigate the partial groups which are exemplars of non-embeddability. 
Finally we show that a partial groupoid embeds in a groupoid if and only if its reduction embeds in a group.  
\end{abstract}

\maketitle

A partial group (in Chermak's sense \cite{Chermak:FSL}) is \emph{embeddable} if it is an impartial subgroup of some group.
There is an evident obstruction to embeddability: if there is a word $w$ in the elements of the partial group and two total parenthesizations of that word such that the successive binary multiplications specified by the two parenthesizations exist and are different, then it is impossible to embed in any group.
We show here that the converse holds, a result essentially going back to Baer (\cref{rmk baer}).
This theorem has been spoken about since the early days of partial groups. 
For example, it was a topic of discussion in July 2011 at the masterclass ``Fusion systems and $p$-local group theory'' in Copenhagen.

The corresponding result without assuming invertibility of elements is false, but the theorem is also true in the many-object case (partial groupoids), and it is no more effort to work at this level of generality.
Our proof uses well-known material about the reflection $\tau\colon \sset \to \cat$ of simplicial sets into categories going back to Gabriel \& Zisman \cite{GabrielZisman:CFHT}.

We also construct a collection of universal counterexamples to embeddability, one for each suitable pair of parenthesizations of a generic length $n$ word.
Using these, the category of embeddable partial groupoids may be concretely described as an orthogonality class in the category of partial groupoids.
We compute the degree, in the sense of \cite{HackneyLynd:HSSPG}, of these universal counterexamples.

This paper uses the formalism of symmetric (simplicial) sets, an enhancement of simplicial sets which is suitable for modeling invertibility.
Symmetric sets have proved to be a convenient home for working with partial groups \cite{HackneyLynd:PGSSS}.
Let $\Upsilon$ be the category whose objects are the sets $[n] = \{0,1,\dots,n\}$ for $n\geq 0$, with arbitrary functions as morphisms.
A symmetric set is just a functor $\Upsilon^\op \to \set$. 
By composing with the inclusion $\Delta \to \Upsilon$, we obtain a functor $\sym \to \sset$ which takes a symmetric set to its underlying simplicial set.

\section{Embeddability}

The nerve functor $N\colon \cat \to \sset$ admits a left adjoint $\tau$ which sends a simplicial set $X$ to the category $\tau X$ which is formed by first taking the free category on the underlying reflexive directed graph and then imposing the relation $(d_0\sigma) \circ (d_2\sigma) = d_1\sigma$ for every $\sigma \in X_2$. 
In particular, $\tau X$ only depends on the 2-skeleton of $X$.

The fundamental groupoid functor is the composite 
\[ \begin{tikzcd}
\sset \rar{\tau} & \cat \rar[bend left] & \gpd, \lar[hook', "\bot"']
\end{tikzcd} \]
where the left adjoint $\cat \to \gpd$ inverts all morphisms.
The fundamental group at a vertex $x\in X_0$ is the set of endomorphisms of $x$ in the fundamental groupoid.
The composite $\sym \to \sset \to \cat$ already lands in the category of groupoids, so the fundamental groupoid of a symmetric set $X$ is just $\tau X$.

If $X$ is a simplicial set, $\words(X)$ is the graded set of compatible words in $X_1$.
That is, $\words(X)_0 = X_0$ and \[ \words(X)_n = X_1 \times_{X_0} X_1 \times_{X_0} \dots \times_{X_0} X_1 = \{ x_0 \xrightarrow{f_1} x_1 \xrightarrow{f_2} \cdots \xrightarrow{f_n} x_n \}.\]
We say $X$ is \emph{edgy} if the Segal map $\edgemap_n \colon X_n \to \words(X)_n$ picking out principal edges is injective for all $n\geq 1$ \cite[Definition 1.7]{HackneyLynd:PGSSS}.
An edgy simplicial set is a kind of partial category, with partially-defined $n$-ary composition given by the span $\words(X)_n \hookleftarrow X_n \to X_1$, whose rightward leg is the composition of inner face maps.
A symmetric set is called \emph{spiny} if its underlying simplicial set is edgy \cite[\S3]{HackneyLynd:PGSSS}; we also call spiny symmetric sets \emph{partial groupoids}.
Each groupoid is a partial groupoid via the fully faithful nerve functor $N\colon \gpd \to \sym$ \cite[4.1]{Grothendieck:TCTGA3}, whose left adjoint is $\tau$.
In this paper a \emph{partial group} is a reduced spiny symmetric set \cite[Theorem~A]{HackneyLynd:PGSSS}.

From now on we are only interested in edgy simplicial sets.
Write $\words^+(X) \subseteq \words(X)$ for the set of positive-length words. 
We define a transitive relation $\rightsquigarrow$ on $\words^+(X)$ using the partially-defined inner face maps \cite[\S4.1]{HackneyLynd:HSSPG}.
Namely, if $w = (f_1, \dots, f_n)$ and there is a 2-simplex $[f_i | f_{i+1}] \in X_2$, we write
\[
	w \rightsquigarrow (f_1, \dots, f_{i+1} \circ f_i, \dots, f_n)
\]
where $f_{i+1} \circ f_i = d_1[f_i|f_{i+1}]$.
The relation $\rightsquigarrow$ is the transitive closure of this operation.
In particular, if $w \rightsquigarrow w'$ then on lengths we have $|w| > |w'|$.
The set of morphisms of $\tau X$ is the quotient of $\words^+(X)$ by the equivalence relation generated by $\rightsquigarrow$.
Write $[w] \subseteq \words^+(X)$ for the equivalence class of $w\in \words^+(X)$.

Suppose $f,g \in X_1$ become equal in $\tau X$.
There is a zig-zag
\[
\begin{tikzcd}[column sep=tiny, row sep=tiny]
& \mathclap{w_1} \drar[rightsquigarrow] \dlar[rightsquigarrow] && \mathclap{w_3} \drar[rightsquigarrow] \dlar[rightsquigarrow] &&[+0.7cm] &&   \mathclap{w_{2n-1}} \drar[rightsquigarrow] \dlar[rightsquigarrow] \\
\mathclap{f} && \mathclap{w_2} && \mathclap{w_4} 
\ar[rr,"\hspace{-0.35em}\cdots", bend left=40, phantom]
&&
\mathclap{w_{2n-2}} && \mathclap{g}
\end{tikzcd}
\]
between $f$ and $g$.
In general, we don't have any control over the length of this zig-zag between $f$ and $g$.
\begin{example}
Consider the 2-dimensional simplicial set with eight nondegenerate 2-simplices and thirteen nondegenerate 1-simplices, as displayed below.
\[ \begin{tikzcd}
&[-1.5em] 1 \rar["b"] & 2 \drar["c"] &[-1.5em] 
&
&[-1.5em] 1 \rar["b"] \ar[drr,"e"'] & 2 \drar["c"] &[-1.5em]
\\
0 \urar["a"] \ar[rrr,"f"'] \ar[urr,"d"'] & & & 3 
& 
0 \urar["a"] \ar[rrr,"h"'] & & & 3
\\[-1em]
& 1' \rar["q"] & 2' \drar["r"] & 
&
& 1' \rar["q"] \ar[drr,"t"'] & 2' \drar["r"] &
\\
0 \urar["p"] \ar[rrr,"h"'] \ar[urr,"s"'] & & & 3 
& 
0 \urar["p"] \ar[rrr,"g"'] & & & 3
\end{tikzcd} \]
This simplicial set is edgy.
We have a zigzag $f \leftsquigarrow (a,b,c) \rightsquigarrow h \leftsquigarrow (p,q,r) \rightsquigarrow g$ but no shorter one will do to relate $f$ and $g$.
We can vary this construction to require a zigzag of arbitrarily high length.
\end{example}

\begin{theorem}\label{thm injective}
Let $X$ be a partial groupoid and $f,g\in X_1$.
If $[f] = [g]$ in $\tau X$, then there is a word $w \in \words^+(X)$ with $f \leftsquigarrow w \rightsquigarrow g$.
\end{theorem}
\begin{proof}
Suppose $n\geq 1$ and we have a zigzag of elements of $\words^+(X)$ as follows.
\begin{equation}\label{eq:zigzag}
\begin{tikzcd}[column sep=tiny, row sep=tiny]
& \mathclap{w_1} \drar[rightsquigarrow] \dlar[rightsquigarrow] && \mathclap{w_3} \drar[rightsquigarrow] \dlar[rightsquigarrow] &&[+0.7cm] &&   \mathclap{w_{2n-1}} \drar[rightsquigarrow] \dlar[rightsquigarrow] \\
\mathclap{w_0} && \mathclap{w_2} && \mathclap{w_4} 
\ar[rr,"\hspace{-0.35em}\cdots", bend left=40, phantom]
&&
\mathclap{w_{2n-2}} && \mathclap{w_{2n}}
\end{tikzcd}
\end{equation}
Set
\[
	w = \begin{cases}
		w_1 w_3^{-1} w_5 w_7^{-1} \cdots w_{2n-3}^{-1} w_{2n-1} & \text{if $n$ is odd} \\
		w_1 w_3^{-1} w_5 w_7^{-1} \cdots w_{2n-1}^{-1} w_{2n} & \text{if $n$ is even}
	\end{cases}
\]
and consider the `odd reduction' that replaces all $w_{2k-1} w_{2k+1}^{-1}$ with $w_{2k} w_{2k}^{-1}$ for $1 \leq k < n$ odd, and the `even reduction' that replaces all $w_{2k-1}^{-1}w_{2k+1}$ with $w_{2k}^{-1} w_{2k}$ for $2 \leq k < n$ even.
Suppose $n$ is odd.
Then the odd reduction of $w$ takes the form
\[
	w_1 w_3^{-1} w_5 w_7^{-1} \cdots w_{2n-3}^{-1} w_{2n-1} \rightsquigarrow w_2 w_2^{-1} w_6 w_6^{-1} \cdots w_{2n-4}^{-1} w_{2n-1}
\]
which further reduces to $w_{2n-1}$.
The even reduction of $w$ takes the form 
\[
	w_1 w_3^{-1} w_5 w_7^{-1} \cdots w_{2n-3}^{-1} w_{2n-1} \rightsquigarrow w_1 w_4^{-1} w_4 w_8^{-1} \cdots w_{2n-2}^{-1} w_{2n-2}
\]
which further reduces to $w_1$.
Thus when $n$ is odd we have
\[
	w_0 \leftsquigarrow w_1 \leftsquigarrow w \rightsquigarrow w_{2n-1} \rightsquigarrow w_{2n}.
\]

Next suppose $n$ is even.
The odd reduction of $w$ takes the form 
\[
	w_1 w_3^{-1} w_5 w_7^{-1} \cdots w_{2n-3} w_{2n-1}^{-1} w_{2n} \rightsquigarrow w_2 w_2^{-1} w_6 w_6^{-1} \cdots w_{2n-2} w_{2n-2}^{-1} w_{2n}
\]
which further reduces to $w_{2n}$.
The even reduction of $w$ takes the form 
\[
	w_1 w_3^{-1} w_5 w_7^{-1} \cdots w_{2n-3} w_{2n-1}^{-1} w_{2n} \rightsquigarrow w_1 w_4^{-1} w_4 w_8^{-1} \cdots w_{2n-4}^{-1} w_{2n-1}^{-1} w_{2n}
\]
which further reduces to $w_1 w_{2n-1}^{-1} w_{2n} \rightsquigarrow w_1 w_{2n}^{-1} w_{2n} \rightsquigarrow w_1$.
Thus when $n$ is even we have 
\[
	w_0 \leftsquigarrow w_1 \leftsquigarrow w \rightsquigarrow w_{2n}.
\]

If $[f] = [g]$ then we have a zig-zag as in \eqref{eq:zigzag} with $w_0 = f$ and $w_{2n} = g$ (noting that when $f = g \colon a \to b$ we may use the word $w_1 = (\id_a, f)$).
We then have $f \leftsquigarrow w \rightsquigarrow g$.
\end{proof}

Recall the partially defined inner face maps $d_i \colon \words(X)_n \nrightarrow \words(X)_{n-1}$ for $0 < i < n$.
A word $w \in \words(X)_n$ is \emph{mean} if there are sequences $i_1, \dots, i_{n-1}$ and $j_1, \dots, j_{n-1}$ with $0<i_k,j_k < k+1$ such that $d_{i_1} \dots d_{i_{n-1}} (w)$ and $d_{j_1} \dots d_{j_{n-1}} (w)$ are defined and are different, and \emph{kind} otherwise. 
(For example, $d_1d_2d_3(w) \neq d_1d_2d_2(w)$.)
We call each one-simplex arising from iterated inner faces of a mean word \emph{sad}.
All other one-simplices are called \emph{happy}. 

\begin{corollary}\label{embeddability theorem}
Let $X$ be a partial groupoid.
The following are equivalent:
\begin{enumerate}
\item $X$ embeds in a groupoid.\label{item gpd}
\item $X$ embeds in its fundamental groupoid via the unit map $X \to N\tau X$. \label{item fundgpd}
\item All one-simplices are happy.\label{item happy}
\item All words are kind.\label{item kind}
\end{enumerate}
\end{corollary}
\begin{proof}
If $X$ embeds in a groupoid, i.e. if $X \subseteq NG$ for some groupoid $G$, then the inclusion is equal to $X \to N\tau(X) \to N\tau (NG) = NG$, so $X$ embeds in its fundamental groupoid.
Spininess implies that embeddability of $X$ in $N\tau X$ can be detected at the level of 1-simplices.
By the previous theorem, $f\in X_1$ is happy if and only if $[f]\cap X_1 = \{f\}$.
Thus $X$ embeds in its fundamental groupoid if and only if all 1-simplices are happy.
\end{proof}

\begin{remark}\label{rmk baer}
\Cref{thm injective} and \cref{embeddability theorem} are many-object versions of Corollary 1 and Corollary 2 of \cite[\S4]{Baer:FSGG}.
Indeed, a 2-skeletal partial group \cite{HackneyMolinier:DPG} is essentially the same thing as what Baer termed an `add which is self-reflexive in the strict sense':  
in the forward direction one sends $X$ to $X_1$ equipped with its partial binary multiplication, while in the reverse direction one builds a partial group whose set of 2-simplices is the set of multipliable pairs (for details see \cite{HackneyLyndSalati:PGPG}).
Baer's theorems were later rediscovered by Tamari \cite[\S III]{Tamari:PAMPMG}, who called these partial magmas \emph{monoïdes symétriques}. The one-object version of \cref{thm injective} was later known as Tamari's One Mountain Theorem, e.g. \cite[Theorem 2.a]{BLT78}. 
\end{remark}

\begin{definition}
A partial groupoid $X$ is \emph{embeddable} if $X \to N\tau X$ is a monomorphism.
\end{definition}

For a partial group $X$ this definition agrees with that from the first sentence of the article, as an impartial subgroup of $X$ is precisely a nonempty symmetric subset.

\begin{theorem}\label{theorem reflective}
The category of embeddable partial groupoids is a reflective subcategory of the category of partial groupoids.
(Similarly for partial groups.)
\end{theorem}
\begin{proof}
Use the epi-mono factorization of the unit natural transformation $\eta \colon \id \Rightarrow N\tau$, 
i.e. the reflection is given as the first map $X \twoheadrightarrow \overline{X} \rightarrowtail N\tau X$.
\end{proof}

\begin{example}
Consider the 2-dimensional simplicial set with eight nondegenerate 2-simplices and thirteen nondegenerate edges, as depicted below.
\[ \begin{tikzcd}
&[-1.5em] 1 \rar["b"] & 2 \drar["c"] &[-1.5em] 
&
&[-1.5em] 1 \rar["b"] \drar["p" description] \dar["m"'] & 2 \dar["c"] &[-1.5em]
\\
0 \urar["a"] \ar[rrr,"f"'] \ar[urr,"d"'] & & & 3 
& 
& 2' \rar["c'"']& 3 & 
\\[-1em]
& 1' \rar["b'"] \ar[drr,"t"'] & 2' \drar["c'"] & 
&
& 1 \rar["m"]  & 2' &
\\
0 \urar["a'"] \ar[rrr,"g"'] & & & 3 
& 
 & 0 \uar["a"] \rar["a'"'] \ar[ur,"q" description] & 1' \uar["b'"'] & 
\end{tikzcd} \]
This simplicial set is edgy.
Notice 
\[ f \leftsquigarrow (a,b,c) \rightsquigarrow (a,p) \leftsquigarrow (a,m,c') \rightsquigarrow (q,c') \leftsquigarrow (a',b',c') \rightsquigarrow g,\]
so $f$ and $g$ represent the same morphism in $\tau X$, but all words are kind.
Indeed, $f$ and $g$ are the only parallel edges, hence the only potentially sad edges, but none of the four relations $f \leftsquigarrow (a,m,c') \rightsquigarrow g$,  $(a,b,c) \rightsquigarrow g$, $(a',b',c') \rightsquigarrow f$ hold.
Thus the simplicial analogue of \cref{thm injective} is false.
\end{example}

\begin{example}\label{ex pregroup}
The nerve of a pregroup (in the sense of Stallings) is embeddable \cite{LemoineMolinier:PGPRFS,Stallings:GT3DM}.
Pregroups, as well as known generalizations (e.g.\ \cite{Hoare:GSP,KushnerLipschutz:GSP}), satisfy an associativity axiom stating that if $ab$ and $bc$ are defined, then $(ab)c$ or $a(bc)$ defined implies $(ab)c=a(bc)$.
This axiom need not hold in an embeddable partial group -- for instance it does not hold for the reduction and symmetrization of the horn $\Lambda^3_1 \subset \Delta^3$.
\end{example}

\begin{example}
The free partial group on one generator embeds in $B\mathbb{Z}/3\mathbb{Z}$, but its fundamental group is $\mathbb{Z}$. 
Thus $\pi_1(X) \to G$ need not be injective when $X$ embeds in $G$.
\end{example}

\begin{remark}\label{remark fusion systems}
Fundamental groups of linking systems of saturated fusion systems are of interest in $p$-local group theory \cite[\S III.7 (8)]{AschbacherKessarOliver2011}. 
Linking systems are special kinds of transporter systems \cite[\S 3]{OliverVentura2007}, which are equivalent to special partial groups called localities \cite{Chermak:FSL}, \cite[\S 2]{GlaubermanLynd2021}. 
The locality $L$ associated with a transporter system $\mathcal{T}$ is a quotient $N\mathcal{T} \twoheadrightarrow L$ of the nerve, and this quotient map is a weak equivalence \cite[Appendix]{gonzalez}. 
\Cref{thm injective} may thus be interpreted in light of the computations of fundamental groups of linking systems that have already been carried out. 
As one example, $\pi_1(L^c_{\mathrm{Sol}}(q)) = 1$ for the centric linking locality of a Benson--Solomon fusion system \cite{ChermakOliverShpectorov:LSS2LFGSC}. 
By \cref{thm injective} any two elements of $L^c_{\mathrm{Sol}}(q)$ have some word associating to each. 
\end{remark}

\section{Orthogonality set}\label{sec ortho set}

In \cite{HackneyLynd:HSSPG} a nonembeddable partial groupoid is constructed, based on the two possible triangulations of the square.
This partial groupoid is universal with respect to the property $(fg)h \neq f(gh)$ among the binary compositions in the partial groupoid.
The analysis of the previous section relied on more complicated parenthesizations of words, and in general we will need more complicated variations of this example to detect nonembeddability.
Recall that (full) parenthesizations of $n$ symbols may equally well be described as a triangulation of an $(n{+}1)$-gon.
These are the elements of the Tamari lattice, or the vertices in a Stasheff associahedron.

\begin{figure}
\includegraphics[width=\textwidth]{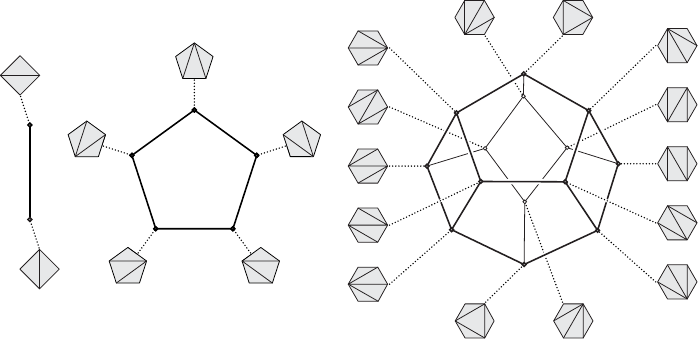}
\caption{Flip graphs showing different triangulations. Image by TheMathCat at English Wikipedia, licensed under CC BY 4.0. \url{https://en.wikipedia.org/wiki/File:Flip_graphs.svg}}\label{flip graphs}
\end{figure}

We now consider triangulations $T$ of a regular $(n{+}1)$-gon having vertex set $\{0,1,\dots,n\}$. 
Each such $T$ determines a symmetric subset
\[
	H(T) \subseteq \sk_2 \rep{n} \subseteq \rep{n}
\]
consisting of precisely those triangles in $T$ (here $\rep{n}$ denotes the representable functor $\hom_\Upsilon(-,[n]) \colon \Upsilon^\op \to \set$). 
Notice that $\spine{n} \subseteq \rep{n}$, consisting of those functions $[m] \to [n]$ whose image lands in $\{i,i+1\}$ for some integer $i$, is a symmetric subset of $H(T)$ for every such triangulation.
We regard $T$ as a set of triples $(i,j,k)$ with $0 \leq i < j < k \leq n$, where $i,j,k$ are the vertices appearing in a triangle.

\begin{proposition}\label{prop generalized NA}
Let $T$ and $T'$ be two triangulations of a regular $(n{+}1)$-gon.
The symmetric set $H(T) \amalg_{\spine{n}} H(T')$ is spiny if and only if for each $i=1, \dots, n-1$, 
\[ (i-1,i,i+1) \notin T\cap T'.\]
\end{proposition}
\begin{proof}
Let $X= H(T) \amalg_{\spine{n}} H(T')$.
If $(i-1,i,i+1) \in T\cap T'$, then there are distinct 2-simplices having spine $i-1 \to i \to i+1$, one coming from $H(T)$ and the other coming from $H(T')$, so the pushout of symmetric sets is not spiny.
	
Suppose $(i-1,i,i+1) \notin T\cap T'$ for any $i$.
To prove that the 2-dimensional symmetric set $X$ is spiny, it suffices to show that $\edgemap_2 \colon X_2 \to X_1 \times_{X_0} X_1$ is injective \cite[Theorem 9]{HackneyMolinier:DPG}.
As the subobjects $H(T)$, $H(T')$ of $X$ are spiny, it suffices to consider $x \in H(T)_2$ and $x' \in H(T')_2$ such that $\edgemap_2(x) = \edgemap_2(x')$ in $X$.
Since the intersection of $H(T)$ and $H(T')$ in $X$ is $\spine{n}$, both edges 
\[ \begin{tikzcd}
k_0 \rar{f} & k_1 \rar{g} & k_2,
\end{tikzcd} \]
appearing in $\edgemap_2(x)$ are in $\spine{n}$, hence $|k_1 - k_0| \leq 1$ and $|k_2 - k_1| \leq 1$.
We thus have $\{k_0, k_1, k_2\} \subseteq \{i-1, i, i+1\}$ for some $i$; by our assumption, equality does not hold.
Further, $\{k_0, k_1, k_2\} \neq \{i-1, i+1\}$ since it would imply either $f$ or $g$ not in $\spine{n}$.
We conclude that $\{k_0, k_1, k_2\}$ is a subset of either $\{i, i- 1\}$ or $\{i,i+1\}$, so $x,x' \in \spine{n}$.
Thus $x=x'$ since $\spine{n}$ is spiny. 
\end{proof}

\begin{definition}
\label{compatible}
If $T$ and $T'$ are two triangulations of the regular $(n{+}1)$-gon such that $H(T) \amalg_{\spine{n}} H(T')$ is spiny, we write $\NA^{T,T'}$ for this partial groupoid.
We say that the pair $T,T'$ is \emph{compatible} in this case.
\end{definition}

See Figures~\ref{fig pentagon not compatible} and \ref{fig pentagon compatible plus NA}.
Note that $\NA$ from \cite[\S2.3]{HackneyLynd:HSSPG} is $\NA^{T,T'}$ for $T,T'$ the two triangulations of a square. 

\begin{figure}
\begin{tikzpicture}[scale=2,thick]
  \node[regular polygon,
        regular polygon sides=5,
        minimum size=2cm,
        draw,
        rotate=144] (P1) at (0,0) {};

  \foreach \i in {1,...,5} {
    \node[inner sep=1pt] at ($ (P1.center)!1.2!(P1.corner \i) $)
      {\the\numexpr\i-1\relax};
  }

\begin{scope}
  \clip (P1.corner 1) -- (P1.corner 2) -- (P1.corner 3) -- 
        (P1.corner 4) -- (P1.corner 5) -- cycle;  
  \draw (P1.corner 5) -- (P1.corner 2);
  \draw (P1.corner 5) -- (P1.corner 3);
\end{scope} 

  \node[regular polygon,
        regular polygon sides=5,
        minimum size=2cm,
        draw,
        rotate=144] (P2) at (1.5,0) {};

  \foreach \i in {1,...,5} {
    \node[inner sep=1pt] at ($ (P2.center)!1.2!(P2.corner \i) $)
      {\the\numexpr\i-1\relax};
  }

\begin{scope}
  \clip (P2.corner 1) -- (P2.corner 2) -- (P2.corner 3) -- 
        (P2.corner 4) -- (P2.corner 5) -- cycle;  
  \draw (P2.corner 3) -- (P2.corner 1);
  \draw (P2.corner 3) -- (P2.corner 5);
\end{scope}

\end{tikzpicture}
\caption{Incompatible triangulations of the pentagon}\label{fig pentagon not compatible}
\end{figure}
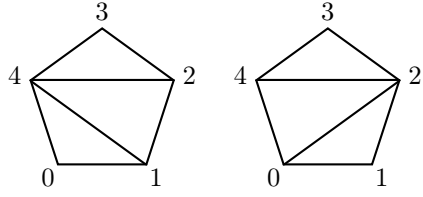

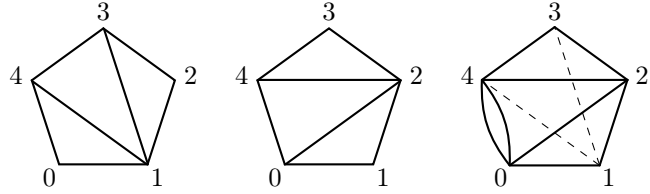
\begin{figure}
\begin{tikzpicture}[scale=2,thick]
  \node[regular polygon,
        regular polygon sides=5,
        minimum size=2cm,
        draw,
        rotate=144] (P1) at (0,0) {};

  \foreach \i in {1,...,5} {
    \node[inner sep=1pt] at ($ (P1.center)!1.2!(P1.corner \i) $)
      {\the\numexpr\i-1\relax};
  }

\begin{scope}
  \clip (P1.corner 1) -- (P1.corner 2) -- (P1.corner 3) -- 
        (P1.corner 4) -- (P1.corner 5) -- cycle;  
  \draw (P1.corner 2) -- (P1.corner 4);
  \draw (P1.corner 2) -- (P1.corner 5);
\end{scope} 

  \node[regular polygon,
        regular polygon sides=5,
        minimum size=2cm,
        draw,
        rotate=144] (P2) at (1.5,0) {};

  \foreach \i in {1,...,5} {
    \node[inner sep=1pt] at ($ (P2.center)!1.2!(P2.corner \i) $)
      {\the\numexpr\i-1\relax};
  }

\begin{scope}
  \clip (P2.corner 1) -- (P2.corner 2) -- (P2.corner 3) -- 
        (P2.corner 4) -- (P2.corner 5) -- cycle;  
  \draw (P2.corner 3) -- (P2.corner 1);
  \draw (P2.corner 3) -- (P2.corner 5);
\end{scope}

\node[regular polygon,
      regular polygon sides=5,
      minimum size=2cm,
      rotate=144,
      draw=none] (P) at (3,0) {};

\foreach \i in {1,...,5} {
  \coordinate (\the\numexpr\i-1\relax) at (P.corner \i);
}

\foreach \i in {1,...,5} {
  \node[inner sep=1pt] at ($ (P.center)!1.2!(P.corner \i) $)
    {\the\numexpr\i-1\relax};
}


\begin{scope}[line join=round, line cap=round]
\draw (0)--(1);
\draw (1)--(2);
\draw (2)--(3);
\draw (3)--(4);

\draw (0)--(2);
\draw (2)--(4);
\end{scope}

\draw[dashed, thin] (1)--(3);
\draw[dashed, thin] (1)--(4);

\draw[bend right=20] (0) to (4);
\draw[bend left=20] (0) to (4);

\end{tikzpicture}
\caption{Compatible triangulations of the pentagon and $\NA^{T,T'}$}\label{fig pentagon compatible plus NA}
\end{figure}

\begin{proposition}\label{prop iden long edges}
A partial groupoid $X$ is embeddable if and only if for each compatible pair $T,T'$ of triangulations of an $(n{+}1)$-gon, any map $\NA^{T,T'} \to X$ identifies the two long edges $0\to n$.
\end{proposition}
\begin{proof}
Suppose a map $\NA^{T,T'} \to X$ sends the long edges to $f,g \in X_1$. 
The image of $\spine{n} \subset \NA^{T,T'}$ produces a word $w\in \words(X)_n$ such that $f \leftsquigarrow w \rightsquigarrow g$, so $[f] = [g] \in \tau X$.
If $X$ is embeddable, then $f = g$.

Conversely, suppose $f\neq g\in X_1$ with $[f] = [g]$.
By \cref{thm injective} there is a word $w$ with $f \leftsquigarrow w \rightsquigarrow g$.
Assume $w$ has the minimal length (necessarily greater than two) with respect to this property.
The parenthesizations of $w = u_1 \cdots u_k$ yielding $f$ and $g$ cannot both contain $(u_iu_{i+1})$ for any $i = 1,\dots,k-1$, or there would exist a word of shorter length.
Letting $T$ and $T'$ be the triangulations associated to the two parenthesizations, we thus have $T,T'$ is a compatible pair by \cref{prop generalized NA}.
By design, we have a map $\NA^{T,T'} \to X$ which takes the spine to $w$ and the two long edges to $f,g$.
\end{proof}

Let $\circspine{n} \subseteq \sk_1\rep{n}$ be the usual spine together with the long edge $0\to n$ (the \emph{circular spine}).
A slight variation on the proof of \cref{prop generalized NA} gives the following.

\begin{proposition}
Let $T$ and $T'$ be two triangulations of a regular $(n{+}1)$-gon.
The symmetric set $H(T) \amalg_{\circspine{n}} H(T')$ is spiny if and only $H(T) \amalg_{\spine{n}} H(T')$ is spiny and
\[ (n-1,n,0), (n,0,1) \notin T\cap T'.\]
\end{proposition}

\begin{definition}
\label{well behaved}
If $T$ and $T'$ are two triangulations of the regular $(n{+}1)$-gon such that $H(T) \amalg_{\circspine{n}} H(T')$ is spiny, we say the pair $T,T'$ is \emph{well-behaved} and write $\NNA^{T,T'}$ for this partial groupoid.
\end{definition}

Notice that if $n>3$ and $T,T'$ are adjacent in the flip graph (\cref{flip graphs}), then necessarily the pair is either not compatible or not well-behaved.

\begin{lemma}\label{lem cut edges}
Let $X$ be a partial groupoid, $g,g',h,h',f \in X_1$. 
Suppose $(f,g) \rightsquigarrow h$ and $(f,g') \rightsquigarrow h'$.
Then $g=g'$ if and only if $h=h'$. (Similarly if $(g,f) \rightsquigarrow h$ and $(g',f) \rightsquigarrow h'$.)
\end{lemma}
\begin{proof}
The forward direction is clear.
For the reverse, the assumption implies $[f|g]$ and $[f|g']$ are 2-simplices, hence $[f^{-1} | f | g]$ and  $[f^{-1} | f | g']$ are 3-simplices.
But then we have 
\[
	g = d_1 d_1 [f^{-1}|f|g] = d_1 d_2 [f^{-1}|f|g] = d_1[f^{-1}|h]
\]
and similarly $g' = d_1[f^{-1}|h']$, so  $h=h'$ implies $g=g'$.
\end{proof}

If $C$ is a category, then an object $x \in C$ is \emph{orthogonal} to a map $f\colon a \to b$ if for each map $g\colon a \to x$ there is a unique map $h \colon b \to x$ with $g=hf$.
The following is a refinement of \cref{theorem reflective} by \cite[Theorem 1.39]{AdamekRosicky:LPAC}.

\begin{proposition}
A partial groupoid $X$ is embeddable if and only if it is orthogonal to the set of maps 
\[ \left\{\NA^{T,T'} \to \NNA^{T,T'} \mid T,T' \text{ well-behaved} \right\}.\]
\end{proposition}
\begin{proof}
Suppose $T,T'$ is a compatible pair of triangulations of an $(n{+}1)$-gon (in particular $n > 2$), and suppose we have a map $\NA^{T,T'} \to X$ with $T,T'$ not well-behaved (we may then conclude $n > 3$).
We either have $(n-1,n,0) \in T \cap T'$ or $(n,0,1) \in T\cap T'$.
Deleting the vertex $n$ in the first case, or the vertex $0$ in the second, yields a compatible pair of triangulations $S,S'$ of an $n$-gon, along with the composite map
\[
	\NA^{S,S'} \to \NA^{T,T'} \to X.
\]
By \cref{lem cut edges}, $\NA^{S,S'} \to X$ identifies long edges if and only if $\NA^{T,T'} \to X$ identifies long edges.
We may continue this process until we arrive at a well-behaved pair.
Thus the statement of \cref{prop iden long edges} holds with `compatible pair' replaced by `well-behaved pair.'
But for a well-behaved pair $T,T'$, $\NA^{T,T'} \to X$ identifies the long edges if and only if it factors through $\NNA^{T,T'}$.
\end{proof}

\section{Degree of \texorpdfstring{$\NA^{T,T'}$}{NATT}}

The \emph{degree} is an invariant of symmetric sets introduced in \cite{HackneyLynd:HSSPG}, which is based on the notion of higher Segal space \cite{Dyckerhoff:CPOC}.
As an example, $\deg(X) = 1$ if and only if $X$ is a groupoid.
In \cite[Example 9.8]{HackneyLynd:HSSPG} it was shown that $\deg(\NA^{T,T'})=3$ for $T,T'$ the two triangulations of the square.

Let $T,T'$ be a compatible pair of triangulations of the regular $(n{+}1)$-gon.
A \emph{cone} from $T$ to $T'$ is a pair of triangles $(i-1,i,k)$ and $(i,i+1,k)$ in $T$ such that $(i-1,i,i+1)$ is a triangle in $T'$.
We say that a compatible pair $(T,T')$ \emph{has a cone} if there is a cone from $T$ to $T'$, or vice-versa.

\begin{proposition}\label{prop no cone}
If a compatible pair of triangulations $(T,T')$ does not have a cone, then $\deg(\NA^{T,T'}) = 2$. Otherwise, $\deg(\NA^{T,T'}) = 3$.
\end{proposition}
\begin{proof}
Let $T,T'$ be a compatible pair of triangulations and $X = \NA^{T,T'}$.
Since $X$ is 2-dimensional, $\deg(X) \leq 3$ by \cite[Theorem 9.6]{HackneyLynd:HSSPG}. 
As $X$ is not embeddable in a groupoid, it cannot be a groupoid, so $1 < \deg(X)$.
We write $\bswords(X)_n$ for the set of tuples $(f_1, \dots, f_n) \in \prod X_1$ such that each $f_i$ has the same source; since $X$ is edgy there is an injection $\bous_n \colon X_n \hookrightarrow \bswords(X)_n$ (see \cite[\S2.2]{HackneyLynd:HSSPG}).
By \cite[Proposition 4.10]{HackneyLynd:HSSPG}, $\bous_n$ is surjective for every $n \geq k +1$ if and only if $\deg(X) \leq k$.
We show that $X$ has a cone if and only if $\deg(X) = 3$.

If $(T,T')$ has a cone then there is a starry word $x = (f_1, f_2, f_3) \in \bswords(X)_3$ as pictured such that $d_1x, d_2x, d_3x \in X_2$.
\[ \begin{tikzcd}[row sep=tiny]
& i-1 \\
i \drar["f_3"'] \urar["f_1"] \rar["f_2" description]& k \\
& i+1
\end{tikzcd} \]
This starry word $x$ is not in $X_3$, since if it were then it would be nondegenerate \cite[Lemma~2.13]{HackneyLynd:HSSPG}, and $X$ is 2-dimensional.
Thus $X$ is not lower $3$-Segal, hence $\deg(X) > 2$.

Conversely, suppose that $X$ is not lower $3$-Segal, i.e. $\deg(X) \geq 3$.
If we have a starry word $x = (f_1, \dots, f_n)$ with $f_i = \id_a$,
then $x\in X_n$ if and only if $d_ix = (f_1, \dots, \hat f_i, \dots, f_n) \in X_{n-1}$.
Indeed, if $\sigma \colon [n] \to [n-1]$ is the zero-preserving map defined by 
\[
	\sigma(t) = \begin{cases}
		0 & t = 0,i \\
		t & 1 \leq t \leq i-1 \\
		t-1 & i+1 \leq t \leq n
	\end{cases}
\]
then $\sigma^* d_i x = x$ when $f_i = \id_a$.
A similar argument shows that if $f_i = f_j$, then $d_ix \in X_{n-1}$ if and only if $x\in X_{n}$.

Since $X$ is not lower $3$-Segal, we can choose $n$ minimal such that there exists $x = (f_1, \dots, f_n) \in \bswords(X)_n$ with $d_1x, d_2x, d_3x \in X_{n-1}$, but $x\notin X_n$ \cite[Proposition~4.10]{HackneyLynd:HSSPG}. 
By the above remarks, no $f_i$ is an identity, and  $f_1, f_2$, and $f_3$ each appear exactly once among the $f_i$. 
\[ \begin{tikzcd}[row sep=tiny]
& i_1 \\
a \drar["f_3"'] \urar["f_1"] \rar["f_2" description]& i_2 \\
& i_3
\end{tikzcd} \]
The starry words $(f_1,f_2)$, $(f_1,f_3)$, and $(f_2,f_3)$ are in $X_2$.
We thus have the triangles $(a,i_1, i_2)$, $(a,i_1,i_3)$, and $(a,i_2,i_3)$ in $T\cup T'$.
They cannot all be in $T$ (resp. $T'$) since $T$ is a triangulation.
Without loss of generality, suppose $(f_1,f_2), (f_2,f_3) \in H(T)$ and $(f_1, f_3) \in H(T')$.
Then $f_1$ and $f_3$ are in $H(T) \cap H(T') = \spine{n}$, so we have $i_1 = a \pm 1$ and $i_3 = a \mp 1$.
We've thus constructed a cone.
\end{proof}

\newcommand{\ternassoc}{$\diamondsuit$}

Below we give an alternative characterization of degree in terms of the ternary associativity condition mentioned in \cref{ex pregroup}: 
a partial groupoid $X$ satisfies \ternassoc{} if its underlying partial binary composition has the property that whenever $ab$ and $bc$ are defined, then $(ab)c$ is defined if and only if $a(bc)$ is defined, with equality when both hold.

\begin{lemma}\label{lem assoc ternary vs degree}
If a partial groupoid $X$ does not satisfy \ternassoc, then $\deg(X) > 2$.
\end{lemma}
\begin{proof}
If $X$ does not satisfy \ternassoc, then there is a word $w=(c,b,a) \in \words(X)_3$ such that $ab$, $bc$, and $a(bc)$ exist, but either $(ab)c$ does not exist or it does and $(ab)c \neq a(bc)$.
The faces $d_0w = [b|a]$, $d_1w = [bc|a]$, and $d_3w = [c|b]$ are in $X_2$, but $w\notin X_3$. 
Thus $X$ is not lower 3-Segal by \cite[Remark 3.17]{HackneyLynd:HSSPG}.
\end{proof}

If $G$ is a group with at least four elements, then the 2-skeleton of $BG$ satisfies \ternassoc{} and has degree 3, so the converse of the lemma does not always hold.

\begin{proposition}
Let $(T,T')$ be a compatible pair of triangulations of a regular $(n+1)$-gon and $X = \NA^{T,T'}$. 
Then $X$ satisfies \ternassoc{} if and only if $\deg(X) = 2$.
\end{proposition}
\begin{proof}
The backwards direction follows immediately from \cref{lem assoc ternary vs degree}.
We use \cref{prop no cone} for the forward direction: if $X$ has degree 3, then $(T,T')$ has a cone, and without loss of generality we assume there are triangles $(i-1,i,k), (i,i+1,k)$ in $T$ and $(i-1,i,i+1)$ in $T'$.
There are edges $c\colon i-1 \to i$ and $b\colon i \to i+1$ in $\spine{n}$, $a\colon i+1 \to k$, $ab \colon i \to k$, and $(ab)c \colon i-1 \to k$ in $H(T)$, and $bc \colon i-1 \to i+1$ in $H(T')$.
If $a(bc)$ does not exist, then \ternassoc{} does not hold.
Suppose $a(bc)$ exists.
As $bc$ is not in $\spine{n}$, hence not in $H(T)$, it must be the case that $a\in \spine{n}$.
Since $k\neq i,i+1$, we have $k=i+2$ and $(i-1,i+1,i+2)$ is a triangle of $T'$.
But in this case $a(bc) \colon i-1 \to i+2$ in $H(T')$ is necessarily different from $(ab)c \colon i-1 \to i+2$ in $H(T)$, so again \ternassoc{} fails.
\end{proof}

\section{Reduction and embeddability}

The reduction of a simplicial or symmetric set identifies all of the vertices; if $X$ is a partial groupoid, then its reduction is a partial group \cite[Proposition 5.3]{HackneyLynd:PGSSS}.
\Cref{thm injective} has the following consequence.
Later we establish the converse.

\begin{lemma}\label{lem red emb}
Let $X$ be a partial groupoid. If the reduction of $X$ is an embeddable partial group, then $X$ is embeddable.
\end{lemma}
\begin{proof}
Given a word $w\in \words^+(X)_n$, if $f \leftsquigarrow w \rightsquigarrow g$ then the source of $f$ and $g$ are equal.
In particular, if they are both identities, then $f=g$.
If $X$ is not embeddable, then there exists a mean word $w$.
By the preceding sentence, the reduction of $w$ is still mean.
Thus the reduction $\mathcal{R}X$ of $X$ is not embeddable.
\end{proof}

Suppose $C$ is a category and define a monoid $M = M(C)$ whose elements are strings $(f_1, \dots, f_n)$ for $n\geq 0$ where 
\begin{enumerate}
\item $f_i \colon a_{i-1} \to b_i$ is a nonidentity morphism of $C$, and \label{item reduced}
\item $a_i \neq b_i$ for each $i$. \label{item jagged}
\end{enumerate}
As usual, a string satisfying \eqref{item reduced} is \emph{reduced}, and we call a string satisfying condition \eqref{item jagged} \emph{jagged}.

If we have an arbitrary word which is not reduced, we may delete any identities to make it reduced.
Likewise, given a non-jagged word there is a serration process which uses composition to arrive at a jagged word. 
Of course the reduction of a jagged word may not be jagged, and the serration of a reduced word may not be reduced.
Letting $\rs$ be this iterated reduction-serration procedure, the multiplication of two elements $x = (f_1, \dots, f_n)$ and $y = (g_1, \dots, g_m)$ is defined to be \[ x\bullet y \coloneqq
\rs(yx) = \rs(g_1, \dots, g_m, f_1, \dots, f_n).\]

For convenience, we rephrase the definition of this monoid in terms of rewrite systems and confluence \cite[\S2.1--2.2]{Huet:CR}.
Set $S$ be the set of words in $C_1$, and consider the following rewrite rule $\rightarrow$ on $S$: 
\begin{align*}
	(f_1, \dots, f_n) &\rightarrow (f_1, \dots, f_{i+1} \circ f_i, \dots, f_n) \\
	(f_1, \dots, f_{i-1}, \id_a, f_i, \dots, f_n) &\rightarrow (f_1, \dots, f_{i-1}, f_i, \dots, f_n)
\end{align*}
with the first occurring whenever $f_i$ and $f_{i+1}$ are composable.
As usual the reflexive-transitive closure of $\rightarrow$ is denoted by $\overset{*}\rightarrow$.
The elements of $M$ are precisely the minimal elements of $S$.

\begin{lemma}
Each element of $(S,\rightarrow)$ has a unique normal form.
\end{lemma}
\begin{proof}
The rewriting system $(S,\rightarrow)$ is terminating, since $S$ is graded by length and $\rightarrow$ is strictly decreasing on length.
We show below that $(S,\rightarrow)$ is locally confluent, i.e.\ $x \leftarrow w \rightarrow y$ implies there exists $w'$ such that $x \overset{*}\rightarrow w' \overset{*}\leftarrow y$.
Newman's lemma says that a locally confluent, terminating rewrite system is confluent, i.e.\ $x \overset{*}\leftarrow w \overset{*}\rightarrow y$ implies there exists $w'$ such that $x \overset{*}\rightarrow w' \overset{*}\leftarrow y$.
In particular, each element of $S$ has a unique normal form.

To prove that $(S,\rightarrow)$ is locally confluent, it suffices to show given $x \leftarrow w \rightarrow y$ that $x=y$ or there exists $z$ with $x \rightarrow z \leftarrow y$.
Write $w = (f_1, \dots, f_n)$ and assume $x\neq y$.
As there are two types of $\rightarrow$, we have three possibilities to consider.
First, suppose $x = (f_1, \dots, \hat f_i, \dots, f_n)$ and $y = (f_1, \dots, \hat f_j, \dots, f_n)$ with $i < j$ and $f_i, f_j$ are identities.
Then \[ x \rightarrow (f_1, \dots, \hat f_i, \dots, \hat f_j, \dots, f_n) \leftarrow y.\]

Next suppose $x = (f_1, \dots, \hat f_i, \dots, f_n)$ and $y = (f_1, \dots, f_{j+1} \circ f_j, \dots, f_n)$ where $f_i$ is an identity.
We cannot have $i=j,j+1$, or else we'd have $x =y$.
Setting 
\[
	z = \begin{cases}
		(f_1, \dots, \hat f_i, \dots, f_{j+1} \circ f_j, \dots, f_n) & i < j \\
		(f_1, \dots,  f_{j+1} \circ f_j, \dots, \hat f_i, \dots, f_n) & i > j+1
	\end{cases}
\]
we have $x \rightarrow z \leftarrow y$.
Finally, suppose $i < j$ and $x = (f_1, \dots, f_{i+1} \circ f_i, \dots, f_n)$ and $y = (f_1, \dots, f_{j+1} \circ f_j, \dots, f_n)$.
Setting
\[
	z = \begin{cases}
		(f_1, \dots, f_{i+1} \circ f_i, \dots, f_{j+1} \circ f_j, \dots, f_n) & i \leq j-2 \\
		(f_1, \dots, f_{i+2} \circ f_{i+1} \circ f_i, \dots, f_n) & i = j-1
	\end{cases}
\]
we have $x \rightarrow z \leftarrow y$.
\end{proof}

\begin{lemma}
Multiplication in $M$ is associative.
\end{lemma}
\begin{proof}

Given elements $x,y,z \in M$, we have
\[
\begin{gathered}
	xyz \overset{*}\rightarrow \rs(xy)z \overset{*}\rightarrow \rs(\rs(xy)z)
	\\
	xyz \overset{*}\rightarrow x\rs(yz) \overset{*}\rightarrow \rs(x\rs(yz)) 
\end{gathered}
\]
and so the terms on the right are normal forms of $xyz$.
By the previous lemma, \[ z \bullet (y \bullet x)  = \rs(\rs(xy)z) = \rs(x\rs(yz))  = (z \bullet y)\bullet x. \qedhere\]
\end{proof}

\begin{proposition}
There is a functor $C \to M$ sending a nonidentity element $f\colon a\to b$ to $(f)$ and $\id_a \colon a \to a$ to $()$.
\end{proposition}
\begin{proof}
The functor may uniformly be described by $f \mapsto \rs(f)$.
Given a composable pair $f,g$ of morphisms in $C$, we have
\[
	(g) \bullet (f) = \rs(f,g) = \rs(g\circ f). \qedhere
\]
\end{proof}

\begin{lemma}\label{lem RC embeds}
If $C$ is a category, then $\mathcal{R}C$ embeds in the monoid $M= M(C)$.
\end{lemma}
\begin{proof}
The functor $C \to M$ factors as $C \to \mathcal{R}C \to M$.
The image of $C_1$ in $M$ is precisely the set of words of length at most 1, and $(\mathcal{R}C)_1 \to M$ is injective. 
\end{proof}

Notice that if $C$ is a groupoid, then $M$ is a group: $(f_1, \dots, f_n)^{-1} = (f_n^{-1}, \dots, f_1^{-1})$.

\begin{theorem}
A partial groupoid $X$ is embeddable if and only if its reduction is so.
\end{theorem}
\begin{proof}
Suppose $X$ is embeddable.
By factoring the composite $X \to N\tau X \to \mathcal{R} N\tau X$ through $\mathcal{R} X$ and applying \cref{lem RC embeds} to the category $\tau X$, we have the following diagram
\[ \begin{tikzcd}
X \rar[hook] \dar & N\tau X \dar  \\
\mathcal{R} X \rar & \mathcal{R} N\tau X \rar[hook] & M(\tau X)
\end{tikzcd} \]
The reduction of an injective map is injective, so $\mathcal{R}X$ embeds in the group $M(\tau X)$.
The converse is \cref{lem red emb}.
\end{proof}

\subsection*{Acknowledgments} 
We thank Andy Chermak and Rémi Molinier for discussions related to this paper.

\bibliographystyle{amsplain}
\bibliography{embed}

\end{document}